\documentclass[letterpaper, 10 pt, conference]{ieeeconf}  

\IEEEoverridecommandlockouts                              
\usepackage{bbm}
\usepackage{mathrsfs}
\usepackage{verbatim}
\usepackage{amsmath}
\usepackage{amsfonts}
\usepackage{bm}
\usepackage{graphicx}
\usepackage{caption}
\usepackage{float}

\usepackage{amsthm}
\usepackage{xcolor}
\usepackage[colorinlistoftodos]{todonotes}
\usepackage{mathtools}
\usepackage{cite}

\theoremstyle{definition}
\newtheorem{assumption}{Assumption}
\newtheorem{definition}{Definition}
\newtheorem{remark}{Remark}

\theoremstyle{plain}

\newtheorem{theorem}{Theorem}
\newtheorem{lemma}{Lemma}

\title{\LARGE \bf
Structural Results for Decentralized Stochastic Control with a Word-of-Mouth Communication}

\author{Aditya Dave, {\itshape{Student Member, IEEE,}} and Andreas A. Malikopoulos, {\itshape{Senior Member, IEEE}} 
	\thanks{This research was supported in part by ARPAE's NEXTCAR program under the award number DE-AR0000796 and by the Delaware Energy Institute (DEI).} %
	\thanks{The authors are with the Department of Mechanical Engineering, University of Delaware, Newark, DE 19716 USA (email: \texttt{adidave@udel.edu; andreas@udel.edu).}} }

\begin{document}

\maketitle
\thispagestyle{empty}

\begin{abstract}

In this paper, we analyze a network of agents that communicate through the ``word of mouth," in which, every agent communicates only with its neighbors. We introduce the prescription approach, present some of its properties and show that it leads to a new information state. We also state preliminary structural results for optimal control strategies in systems that evolve using word-of-mouth communication. The proposed approach can be generalized to analyze several decentralized systems.

\end{abstract}

\section{INTRODUCTION}

As we move to increasingly complex systems \cite{Malikopoulos2015} new decentralized control approaches are needed to optimize the impact on system behavior of the interaction between its entities \cite{Malikopoulos}. Centralized stochastic control has been the ubiquitous approach to control complex systems so far \cite{Malikopoulos2016c}. A key assumption in centralized stochastic control problems is that a singular decision maker perfectly recalls all previous control actions and observations. The information available to an agent when making a decision is called the \textit{information structure} of the system. The centralized information structure is classified as the \textit{classical information structure}.

While centralized systems have been extensively studied \cite{23}, the classical information structure does not apply to many applications involving multiple agents \cite{Malikopoulos2018}. In these applications, all agents simultaneously make a decision based only on their local information and information received through delayed, or costly communication with other agents \cite{Malikopoulos2018d}. Thus, a centralized knowledge of the complete information in the system is infeasible \cite{18}. These information structures are classified as \textit{non-classical information structures}. Ultimately, we end up with multi-stage optimization problems \cite{Malikopoulos2015b}, known as decentralized stochastic control problems.

Decentralized stochastic control has proven to be very challenging as the most common approach to derive centralized optimal control policies, dynamic programming (DP), is not directly applicable to non-classical information structures due to a lack of separation between estimation and control. There are three general approaches in the literature for these problems that use techniques from centralized stochastic control: (1) the person-by-person approach, (2) the designer's approach, and (3) the common information approach. Due to space limitations, it is very difficult to cite all the literature around these three approaches. For more details, the reader may refer to the tutorial by Mahajan et al. \cite{Mahajan2012} and the references therein.

1) In the \textit{person-by-person approach}, the control strategies of all agents except one are arbitrarily fixed. Only the control strategy of the chosen agent is then optimized for this new centralized problem. Repeating this process for all agents allows for the derivation of \textit{structural results} and DP for a person-by-person optimal strategy, that is not globally optimal in general. However, every globally optimal strategy must necessarily be person-by-person optimal. Some applications of this approach can be found in \cite{4,5,11,7,8,6,2,1}.

2) The \textit{designer's approach} takes the point of view of a designer with knowledge of the system model and statistics. The designer's task is the \text{blue}{computationally challenging selection of} the globally optimal control strategy for the system by transforming the problem into a centralized planning problem. Some applications of this approach can be found in \cite{3,19,10,lall_broadcast}.

3) A more recent development in this field is the \textit{common information approach} developed for problems with partial history sharing \cite{17}, and then formalized for general decentralized systems \cite{14}. The solution is derived by reformulating the system from the viewpoint of a fictitious \textit{coordinator} whose task is to prescribe control laws to every agent in the system. Some applications can be found in a variety of problems including symmetric delayed information sharing structures \cite{14}, control sharing information structures \cite{16}, stochastic games with asymmetric information \cite{20}, teams with mean-field sharing \cite{Mahajan2015} and systems with unreliable uplink channels \cite{asghari2018optimal}.There are some earlier papers that used similar ideas in analyzing specific information structures \cite{yuksel2009stochastic, yoshikawa1978decomposition, aicardi1987decentralized, casalino1984partially}.

In this paper, we introduce and study a decentralized system with multiple agents who communicate with \textit{word of mouth}. In a word-of-mouth communication, we have a network of agents, where each agent may directly communicate only with its neighbors in the network. Thus, information from each agent propagates in the network through its neighbors who share it with their neighbors, and so on. Each link in the network has a delay associated with it, which can be thought of as the time it takes for the information to transmit from an agent to its neighbor. This problem has a non-classical information structure because of the delays in communication.

We consider the common information approach to be the standard approach in solving a wide variety of decentralized stochastic control problems, including problems with a word-of-mouth information structure. However, we find that in problems with asymmetric communication, there may not be a lot of common information available to all agents in the system \cite{Aditya_2019}. This has motivated us to continue looking for structural results that can improve on the performance of the common information approach by taking into account the asymmetries in a system. 

The contributions of this paper are:

1) We introduce and analyze a problem with a word-of-mouth information structure.

2) We present the prescription approach and its properties which lead to a reformulation of the problem from the point of view of every agent, with a state sufficient for input output mapping and information state for each reformulation.

3) We state some preliminary structural results with time-invariant domains that arise from the prescription approach.

The rest of the paper is organized as follows. In Section II, we present the problem and the information structure of the system. In Section III, we provide a reformulation of the problem and, in Section IV, we derive the preliminary results for optimal strategies. Finally, in Section V, we draw concluding remarks, and present some ideas for future work. 

\subsection{Notation}

Random variables are denoted by upper case letters and their realization by the corresponding lower case letters. For integers $a<b$, $X_{a:b}$ is shorthand for the vector $(X_a,X_{a+1},\ldots,X_b)$ and $X^{a:b}$ is shorthand for the vector $(X^a,X^{a+1},\ldots,X^b)$. When $a>b$, the dimension of $X^{a:b}$ is 0. The combined notation with $c<d$ and $a<b$, we write $X_{a:b}^{c:d}$ to denote the vector $(X_i^j: i=a,\ldots,b; \, j=c,\ldots,d)$.


For sets $A$ and $B$, $\{A,B\}$ is the set $A \cup B$. For a singleton $\{a\}$ and set $B$, $\{a,B\}$ is the set $\{a\} \cup B$. The function $|\cdot|$ returns the cardinality of a set. The null set is represented by $\emptyset$. We have attempted to use notation consistent with \cite{17} as our work is closely related to it.

The probability and expectation measures that depend on a vector $\boldsymbol{g}$ are written as $\mathbb{P}^{\boldsymbol{g}}(\cdot)$ and $\mathbb{E}^{\boldsymbol{g}}(\cdot)$ respectively.
All equalities involving random variables hold with a probability of 1.

\section{PROBLEM FORMULATION}

\subsection{The Network of Agents}

Consider a network of $K \in \mathbb{N}$ agents represented by a strongly connected  graph $\mathcal{G} =(\mathcal{K}, \mathcal{E})$, where $\mathcal{K} := \{1,\ldots,K\}$ is the set of agents and $\mathcal{E}$ is the set of links. A link from an agent $k \in \mathcal{K}$ to an agent $j \in \mathcal{K}$ is denoted by $(k,j) \in \mathcal{E}$. Every link $(k,j)$ represents a communication link from agent  $k$ to  $j$ which is characterized by a delay of $\delta^{[k,j]} \in \mathbb{N}$ time steps for transferring information from $k$ to $j$.

When agent $k$ sends out information to agent $j$ through link $(k,j)$, we call it \textit{transmission of information}. The information transmitted by agent $k$ at time $t$ is received by agent $j$ at time $t + \delta^{[k,j]}$. For any agent $k$, the acts of receiving and transmission of information occur at different instances within every time step as discussed in Section II-D.

\begin{definition}
\label{path}
Let $\mathcal{N} = \{1,\ldots,m : m \in \mathcal{K}\}$ be a set of indices. For any $k,j \in \mathcal{K}$, a \textit{path}  $q^{[k,j]}_a$, $a \in \mathbb{N}$, from $k$ to $j$ is given by the sequence $\{k_n\}_{n\in\mathcal{N}}$ such that: (1) $k_1 = k$ and $k_m = j$, (2) $k_n \in \mathcal{K}$ for $n \in \mathcal{N}$, and (3) there exists a link $(k_{n-1},k_n) \in \mathcal{E}$ for $n \in \mathcal{N}\setminus\{1\}$.
\end{definition}

The set $\mathcal{Q}^{[k,j]} = \{q^{[k,j]}_a: a =1,\ldots,b; \; b \in \mathbb{N}\}$ includes all paths from agent $k$ to agent $j$.

\begin{definition}
\label{path_delay}
Let agents $k,j \in \mathcal{K}$ with a path $q^{[k,j]}_a$ from $k$ to $j$. The \textit{communication delay} $d^{[k,j]}_a \in \mathbb{N}$ for $q^{[k,j]}_a$ is defined as
\begin{equation*}
    d^{[k,j]}_a = \delta^{[k,k_2]}+\cdots+\delta^{[k_{m-1},j]},
\end{equation*}
where $\delta^{[k_{n-1},k_n]}$ is the delay in information transfer through the link $(k_{n-1},k_n) \in \mathcal{E}$.
\end{definition}

The \textit{information path}, defined formally next, from agent $k$ to agent $j$ in the network is the path with the least possible delay.

\begin{definition}
The \textit{information path} from $k$ to $j$ denoted by $(k \rightarrow j)$ is given by a path $q_{{a}}^{[k,j]} \in \mathcal{Q}^{[k,j]}$ such that,
\begin{gather}
    d^{[k,j]}_{{a}} = \min\left\{d_1^{[k,j]},\ldots,d_b^{[k,j]}\right\},
\end{gather}
where $b := |\mathcal{Q}^{[k,j]}|$.
\end{definition}

The strongly connected nature of the network ensures that there is always an information path $(k \rightarrow j)$ from every other agent $k \in \mathcal{K}$ to every agent $j \in \mathcal{K}$. We denote the associated delay simply by $d^{[k,j]}$ and, by convention, we set $d^{[k,k]} = 0$. Also note that because the links in the network are directed, the delay $d^{[k,j]}$ in communication from $k$ to $j$ is not equal to the delay $d^{[j,k]}$ in communication from $j$ to $k$.

\subsection{System Description}
The network of agents is considered a discrete time system that evolves up to a finite time horizon $T \in \mathbb{N}$. At time $t \in \mathcal{T}$, $\mathcal{T} = \{0,1,\ldots,T\}$, the state of the system $X_t$ takes values in a finite set $\mathcal{X}$ and the control variable $U_t^k$ associated with agent $k \in \mathcal{K}$, takes values in a finite set $\mathcal{U}^k$. Let ${U}_t^{1:K}$ denote the vector $(U_t^1,\ldots,U_t^K)$. Starting at the initial state $X_0$, the evolution of the system follows the state equation
\begin{equation}
X_{t+1}=f_t\left(X_t,U_t^{1:K},W_t\right), \label{st_eq}
\end{equation}
where $W_t$ is the uncontrolled disturbance to the system represented as a random variable taking values in a finite set $\mathcal{W}$. At time $t$ every agent $k$ makes an observation $Y_t^k$, given by
\begin{equation}
Y_t^k=h_t^k(X_t,V_t^k), \label{ob_eq}
\end{equation}
which takes values in a finite set $\mathcal{Y}^k$ through a noisy sensor, where $V_t^k$ takes values in the finite set $\mathcal{V}^k$ and represents the noise in measurement.

Agent $k$ selects a control action $U_t^k$ from the set of feasible control actions $\mathcal{U}_t^k$ as a function of its information structure. The information structure is different for each agent $k \in \mathcal{K}$ because of the means of communication and topology of the network. We discuss the information structure in Section II-D. After each agent $k$ generates a control action $U_t^k$, the system incurs a cost $c_t(X_t,U_t^{1:K}) \in \mathbb{R}$.

\subsection{Assumptions}

In our modeling framework above, we impose the following assumptions:

\begin{assumption}
\label{top_asu}
The network topology is arbitrary, known a priori, and does not change with time.
\end{assumption}

With a known and invariable network topology, every agent can keep track of what information is accessible to other agents in the network.

\begin{assumption}
\label{prim_asu}
The external disturbance $\{W_t: t \in \mathcal{T}\}$ and the noise in measurement $\{V_t^k: t \in \mathcal{T}, \, k \in \mathcal{K}\}$ are sequences of independent random variables that are also independent of each other and of the initial state $X_0$. 
\end{assumption}

The external disturbance, noise in measurement, and initial state are referred to as the \textit{primitive} random variables, and they  have known probability distributions. 

\begin{assumption}
\label{func_asu}
The state functions $(f_{t}: t \in \mathcal{T})$, observation functions $(h_{t}^{k}: t \in \mathcal{T}, \, k \in \mathcal{K})$, the cost functions $(c_{t}: t \in \mathcal{T}),$ and the set of all feasible control policies $G$ are known to all agents.
\end{assumption}

These functions and the set of feasible control policies (explained in Section II-D) form the basis of the decision making problem.

\begin{assumption}
Each agent has perfect recall.
\end{assumption}

Perfect recall of the data from the memory of every agent is an essential assumption for the structural results derived in this paper.

\begin{figure}[ht!]
  \centering
  \captionsetup{justification=centering}
  \includegraphics[width=0.9\linewidth, keepaspectratio]{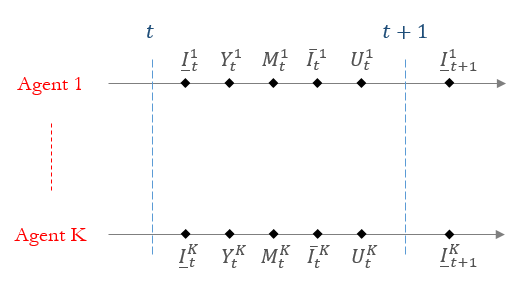}
  \caption{Sequence of activities.}
  \label{fig:update}
\end{figure}

We summarize below the sequence of activities taken by agent $k \in \mathcal{K}$ at time $t$ (Fig. \ref{fig:update}):
\begin{enumerate}
\item The state $X_t$ is updated based on (\ref{st_eq}).
\item Agent $k$ receives information from all agents in $\mathcal{K}$, collectively denoted by $\underline{I}_t^k$.
\item Agent $k$ makes an observation about the state $Y_t^k$ based on (\ref{ob_eq}).
\item Agent $k$ updates its memory, $M_t^k$, defined in Section II-D, on a given protocol.
\item Agent $k$ transmits information denoted by $\overline{I}_t^k$ to every agent $j \in \mathcal{K}$ through the shortest path $q_{{a}}^{[k,j]}$.
\item Agent $k$ generates a control action $U_t^k$.
\end{enumerate}

\subsection{Information Structure of the System}

The information structure of the system is characterized by the graph topology and delays along communication paths described in Section II-A. In the word-of-mouth information structure, every agent $j \in \mathcal{K}$ at time $t$ transmits the information $\overline{I}_t^j := \{Y_{t}^j,U_{t-1}^j\}$ to every other agent in the network through the relevant information paths. Agent $k \in \mathcal{K}$ receives information $\overline{I}_t^j$ at time $t + d^{[j,k]}$, where $d^{[j,k]}$ is the communication delay from $j$ to $k$.
Then, the information available to agent $k$ at time $t$ is the collection of information he received from every agent $j \in \mathcal{K}$ at time steps $0$ through $t$.

\begin{definition}
The \textit{memory} of agent $k \in \mathcal{K}$ is defined as the random variable $M_t^k$ that takes values in the finite set $\mathcal{M}_t^k$ and is given by
\begin{align}
M_t^k := &\left\{Y^j_{0:t-d^{[j,k]}},U_{0:t-d^{[j,k]}-1}^j: j \in \mathcal{K} \right\}, \label{eq_mem}
\end{align}
where $d^{[j,k]}$ is the delay in information transfer from every agent $j \in \mathcal{K}$ to agent $k$.
\end{definition}

At time $t$, agent $k$ accesses his memory $M_t^k$ to generate a control action, namely,
\begin{gather}
    U_t^k := g_t^k(M_t^k), \label{u_basic}
\end{gather}
where $g_t^k$ is the control policy of agent $k$ at time $t$. We define the control policy for each agent as $\boldsymbol{g}^k := (g_0^k,\ldots,g_T^k)$ and the control policy of the system as $\boldsymbol{g} := (\boldsymbol{g}^1,\ldots,\boldsymbol{g}^K)$. The set of all feasible control policies is denoted by $G$.

The performance criterion for the system is given by the total expected cost:
\begin{equation}
\textbf{Problem 1:}~~~~~ \mathcal{J}(\boldsymbol{g}) = \mathbb{E}^{\boldsymbol{g}}\left[\sum_{t=0}^T{c_t(X_t,U_t^{1:K})}\right], \label{per_cri}
\end{equation}
where the expectation is with respect to the joint probability measure on the random variables $\{X_t, U_t^1,\ldots,U_t^K\}$.

The optimization problem is to select the optimal control policy $\boldsymbol{g}^* \in G$ that minimizes the performance criterion in \eqref{per_cri}, given the probability distributions of the primitive random variables $\{X_0,W_{0:T},V_{0:T}^1,\ldots,V_{0:T}^K\}$, and functions $\left\{c_t,f_t,h_t^{k}:t \in \mathcal{T}, \, k \in \mathcal{K} \right\}$.

\section{THE PRESCRIPTION APPROACH}

\subsection{Construction of Prescriptions}

For an agent $k \in \mathcal{K}$, we consider a scenario where the control action $U_t^k$ is generated in two stages:

(1) Agent $k$ generates a function based on information which is a subset of the information available in its memory $M_t^{{k}}$.

(2) This function takes as an input the compliment of the subset used to generate it, and yields the control action $U_t^{{k}}$.

We call these functions \textit{prescriptions}. They allow us to construct an optimization problem of selecting the optimal \textit{prescription strategy} that is equivalent to the problem of selecting the optimal control policy $\boldsymbol{g}^{*k}$ as we show next. In this section, we construct the subset of the memory $M_t^{{k}}$ and prescriptions for every agent ${{k}} \in \mathcal{K}$ without changing the information structure of the system. In order to simplify the notation, we first define the set of agents located beyond agent $k$ in the set of all agents.

\begin{definition}\label{def:setB}
For an agent ${{k}} \in \mathcal{K}$, the set of agents beyond ${{k}}$ is defined as $\mathcal{B}^{{{k}}} := \{{{j}} \in \mathcal{K}: {{j}} \geq {{k}}\}$.
\end{definition}

Now we can define the information used to generate prescriptions.

\begin{definition} \label{def_A}
Let $k \in \mathcal{K}$ and $M_t^{{k}}$ be the agent's memory at time $t$. The \textit{accessible information} of agent ${{k}}$ is defined as the set $A_t^{{k}}$ that takes values in the finite collection of sets $\mathcal{A}_t^{{k}}$ such that
\begin{gather}
    A_t^{{k}} = \bigcap_{{{i}}=1}^{{k}}\left(M_t^{{i}}\right). \label{ainfo_def}
\end{gather}
\end{definition}

For example, we can write \eqref{ainfo_def} for agent $1$ as
\begin{gather}
    A_t^{{1}} = M_t^{{1}},
\end{gather}
and for agent $2$ as
\begin{gather}
    A_t^{{2}} = M_t^{{1}} \cap M_t^{{2}}.
\end{gather}

Based on Definition \ref{def_A}, the accessible information $A_t^{{k}}$ has the following properties:
\begin{align}
    & A_{t-1}^{{k}} \subset A_t^{{k}}, \label{ainfo_prop_2} \\
    & A_t^{{j}} \subset A_t^{{k}}, \; \forall {{j}} \in \mathcal{B}^{{k}}, \label{ainfo_prop_1}
\end{align}
where $\mathcal{B}^{{k}}$ is the set of agents beyond ${{k}}$. 
Property \eqref{ainfo_prop_2} motivates the introduction of a new term to denote the new information added to accessible information $A_t^{{k}}$ at time $t$.

\begin{definition}
The \textit{new information} for agent $k$ at time $t$ is defined as the set $Z_t^{{k}}$ that takes values in a finite collection of sets $\mathcal{Z}_t^{{k}}$ such that
\begin{align}
Z_t^{{k}} := A_t^{{k}} \backslash A^{{k}}_{t-1}. \label{Z_def}
\end{align}
\end{definition}

We observe in \eqref{ainfo_prop_1} that the accessible information $A_t^{{j}}$ of any agent ${{j}} \in \mathcal{B}^{{k}}$ is a subset of the memory $M_t^{{k}}$. Thus, we can define the \textit{inaccessible information} of the agent ${{k}}$ with respect to the accessible information $A_t^{{j}}$ for every ${{j}} \in \mathcal{B}^{{k}}$.

\begin{definition}
The \textit{inaccessible information} of agent ${{k}}$ with respect to accessible information $A_t^{{j}}$, ${{j}} \in \mathcal{B}^{{k}}$, is defined as the set of random variables $L_t^{[{{k}},{{j}}]}$ that takes values in the finite collection of sets $\mathcal{L}_t^{[{{k}},{{j}}]}$ such that
\begin{gather}
    L_t^{[{{k}},{{j}}]} := M_t^{{k}} \setminus A_t^{{j}}. \label{inacc_def}
\end{gather}
\end{definition}


The pair of sets $A_t^{{j}}$ and $L_t^{[{{k}},{{j}}]}$ forms a partition of the set $M_t^{{k}}$, such that
\begin{gather}
    M_t^{{k}} = \{L_t^{[{{k}},{{j}}]},A_t^{{j}}\}, \; \forall {{j}} \in \mathcal{B}^{{k}}. \label{eq_partition}
\end{gather}

As an example, consider a system with three agents Fig. \ref{fig:inaccessible}. In this system, we have
\begin{align}
    &A_t^{1} = M_t^{1}, \nonumber \\
    &A_t^{3} \subset A_t^{2} \subset M_t^{1}, \nonumber \\
    &M_t^{1} = \{A_t^{2},L_t^{[{{1}},{{2}}]}\} = \{A_t^{3},L_t^{[{{1}},{{3}}]}\}.
\end{align}
For agents $2$ and $3$, we can derive similar relationships as illustrated in Fig. \ref{fig:inaccessible}.

\begin{figure}[h!]
  \centering
  \captionsetup{justification=centering}
  \includegraphics[height = 8cm, keepaspectratio]{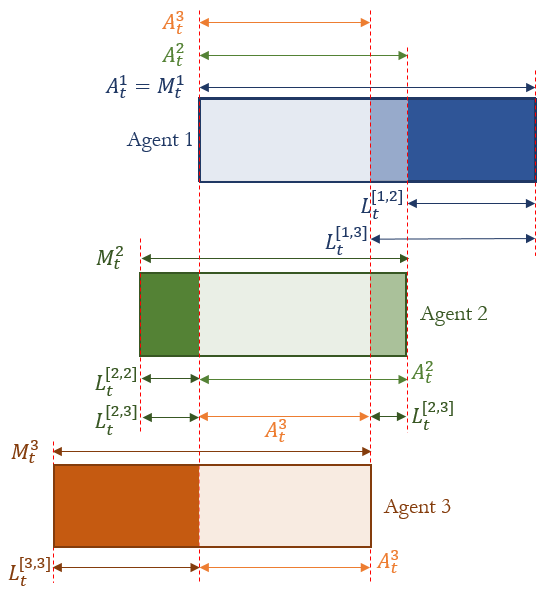}
  \caption{Memory partitions of three agents.}
  \label{fig:inaccessible}
\end{figure}

Next, we use these partitions of the memory to define the \textit{prescription function}.

\begin{definition}
The \textit{prescription function} $\Gamma_t^{[{{k}},{{j}}]}$ of an agent ${{k}} \in \mathcal{K}$ for the agent ${{j}} \in \mathcal{K}$ is defined as follows
\begin{align} \label{pres_func_def}
    \Gamma_t^{[{{k}},{{j}}]} :
    \begin{cases}
        \mathcal{L}_t^{[{{j}},{{k}}]} \longrightarrow \mathcal{U}_t^{{j}},  \text{ if } {{j}} \not\in \mathcal{B}^{{k}}, \\
        \mathcal{L}_t^{[{{j}},{{j}}]} \longrightarrow \mathcal{U}_t^{{j}},  \text{ if } {{j}} \in \mathcal{B}^{{k}},
    \end{cases}
\end{align}
and takes values in the set of feasible prescription functions $\mathscr{G}_t^{[{{k}},{{j}}]}$.
\end{definition}

\begin{remark}
In Definition 9, the inaccessible information of  agent $k$ is defined with respect to the accessible information $A_t^{{j}}$ for ${{j}} \in \mathcal{B}^{{k}}$. Note that in the first part of \eqref{pres_func_def}, we have ${{k}} \in \mathcal{B}^{{j}}$, and thus \eqref{pres_func_def} holds.
\end{remark}

Every prescription function $\Gamma^{[{{k}},{{j}}]}$ is generated as follows
\begin{align} \label{eq_gen_pres}
    \Gamma_t^{[{{k}},{{j}}]} :=
    \begin{cases}
    {\psi}_t^{[{{k}},{{j}}]}(A_t^{{k}}), \text{ if } {{j}} \not\in \mathcal{B}^{{k}}, \\
    {\psi}_t^{[{{k}},{{j}}]}(A_t^{{j}}), \text{ if } {{j}} \in \mathcal{B}^{{k}},
\end{cases}
\end{align}
where we call ${\psi}_t^{[{{k}},{{j}}]}$ the \textit{prescription strategy} of the agent ${{k}}$ for the agent ${{j}}$ given by the mapping
\begin{align}
    \psi_t^{[{{k}},{{j}}]} :
    \begin{cases}
        \mathcal{A}_t^{{{k}}} \longrightarrow \mathscr{G}_t^{[{{k}},{{j}}]},  \text{ if } {{j}} \not\in \mathcal{B}^{{k}}, \\
        \mathcal{A}_t^{{{j}}} \longrightarrow \mathscr{G}_t^{[{{k}},{{j}}]},  \text{ if } {{j}} \in \mathcal{B}^{{k}}.
    \end{cases}
\end{align}

We call $\boldsymbol{\psi}^{{k}} := (\boldsymbol{\psi}^{[{{k}},{{1}}]},\ldots,\boldsymbol{\psi}^{[{{k}},{{K}}]})$  the \textit{prescription strategy} of the agent ${{k}}$. The set of feasible prescription strategies for the agent ${{k}}$ is denoted by $\Psi^{{k}}$. 

\begin{remark}
The prescription $\Gamma_t^{{[{{k}},{{j}}]}}$ of  agent $k$ for agent $j$ is only available to agent $k$. The equivalent prescription available to agent $j$ is $\Gamma_t^{{[{{j}},{{j}}]}}$. The relationship between the two is given in Lemmas \ref{lem_psi_relation1} and \ref{lem_psi_relation2} in Section III-B.
\end{remark}

\begin{remark}
Every agent needs to generate prescriptions corresponding to every other agent in the system so that we can define the \textit{information state} in Section IV-B.
\end{remark}

Next, we define the \textit{complete prescription} of an agent ${{k}}$ below.

\begin{definition}
The \textit{complete prescription} for agent ${{k}}$ is given by the function
\begin{align}
    \Theta_t^{{k}} : \;
    \mathcal{L}_t^{[{{1}},{{k}}]} \times \cdots \times \mathcal{L}_t^{[{{k}},{{k}}]} &\times \mathcal{L}_t^{[{{k+1}},{{k+1}}]} \times \cdots \times \mathcal{L}_t^{[{{K}},{{K}}]} \nonumber \\
    &\longrightarrow \mathcal{U}_t^{{1}} \times \cdots \times \mathcal{U}_t^{{k}},
\end{align}
which takes values in the set of functions $\mathscr{G}_t^{{k}}$.
\end{definition}

The complete prescription for agent ${{k}}$ is constructed as $\Theta_t^{{k}} = (\Gamma_t^{[{{k}},{{1}}]},\ldots,\Gamma_t^{[{{k}},{{K}}]})$.

\subsection{Relationships Between Prescriptions and Control Policies}

In this section we present the relationships between the different prescriptions and control policies.
The first result states that for an agent $k \in \mathcal{K}$ we can use the complete prescription $\Theta_t^{{k}}$ to generate control action $U_t^k$ instead of the control policy $g_t^k$.

\begin{lemma} \label{lem_psi_g_relation}
Let agent $k \in \mathcal{K}$ and let $\Theta_t^{{k}}$ be its complete prescription. For any given control policy $\boldsymbol{g} \in G$, there exists a prescription strategy $\boldsymbol{\psi}^{{k}} \in \Psi^{{k}}$ such that
\begin{gather}
    U_t^{{k}} = \Gamma_t^{[{{k}},{{k}}]}\left(L_t^{[{{k}},{{k}}]}\right). \label{u_presc}
\end{gather}
\end{lemma}

\begin{proof}
Let $A_t^{{k}}$ and $L_t^{[{{k}},{{k}}]}$ be the accessible and inaccessible information, respectively, of agent $k$. For any control policy $\boldsymbol{g}$ that generates $U_t^k$ from \eqref{u_basic}, we can select a prescription strategy $\psi_t^{{k}} : \mathcal{A}_t^{{{k}}} \longrightarrow \mathscr{G}_t^{[{{k}},{{k}}]}$ such that 
\begin{equation}
\Gamma_t^{[{{k}},{{k}}]}(\cdot) = \psi_t^{{k}}(A_t^{{{k}}})(\cdot) =  g_t^{{k}}(A^{{k}}_t,\cdot).
\end{equation}
Then, the control action is
\begin{gather}
    U_t^{{k}} = \Gamma_t^{[{{k}},{{k}}]}(L^{[{{k}},{{k}}]}_t)=g_t^{{k}}(A^{{k}}_t,L^{[{{k}},{{k}}]}_t) = g_t^{{k}}(M_t^{{k}}).
\end{gather}
\end{proof}

Similarly, for any prescription strategy $\boldsymbol{\psi}^{{k}}$, we can construct an appropriate control policy $\boldsymbol{g}$  that generates the same control actions $U_t^{{k}}$ for all agents in $\mathcal{K}$.

\begin{lemma}\label{lem_psi_g_relation_inv}
Let agent $k \in \mathcal{K}$ and let $\Theta_t^{{k}}$ be its complete prescription. For any given prescription strategy $\boldsymbol{\psi}^{{k}} \in \Psi^{{k}}$, there exists a control policy $\boldsymbol{g} \in G$ such that
\begin{gather}
     U_t^{{k}} = \Gamma_t^{[{{k}},{{k}}]}(L_t^{[{{k}},{{k}}]}) = g_t^{{k}}(M_t^{{k}}). \label{u_presc_inv}
\end{gather}
\end{lemma}

\begin{proof}
For any prescription strategy $\boldsymbol{\psi}^{{k}},$ we can construct a control policy $\boldsymbol{g}^{{k}}$ such that
\begin{align}
U_t^{{k}} = g_t^{{k}}(M^{{k}}_t) = &g_t^{{k}}(A^{{k}}_t,L^{[{{k}},{{k}}]}_t) \nonumber \\
= &{\psi}_t^{[{{k}},{{k}}]}(A_t^{{k}})(L^{[{{k}},{{k}}]}_t).
\end{align}
\end{proof}

Lemmas \ref{lem_psi_g_relation} and \ref{lem_psi_g_relation_inv} imply that the control action $U_t^{{k}}$ of every agent $k \in \mathcal{K}$ generated through a prescription strategy $\boldsymbol{\psi}^{{k}}$, can also be generated through an appropriate policy $\boldsymbol{g}$ and vice versa.

\begin{definition}
Given two agents $k,j \in \mathcal{K}$, a \textit{positional relationship} from agent ${{k}}$ to agent ${{j}}$ is given by the function
\begin{gather}
    e^{[{{j}},{{k}}]}: \Psi^{{k}} \longrightarrow \Psi^{{j}}.
\end{gather}
\end{definition}

Next we show the existence of a positional relationship $e^{[{{j}},{{k}}]}$ from any agent ${{k}} \in \mathcal{K}$ to every agent ${{j}} \in \mathcal{K}$ with desirable properties that allow us to construct optimal control policies of all agents from the optimal prescription strategy of just one agent. The following result establishes that using a positional relationship  $e^{[{{j}},{{k}}]}=(e_1^{[{{j}},{{k}}]},\ldots,e_T^{[{{j}},{{k}}]})$, an agent ${{j}}$ can derive the prescription strategy for agent ${{i}} \in \mathcal{K}$, when given the prescription strategy of agent ${{k}}$ for agent ${{i}}$, namely
\begin{gather}
   {\psi}_t^{[{{j}},{{i}}]} := e_t^{[{{j}},{{k}}]}\Big({\psi}_t^{[{{k}},{{i}}]}\Big), \; \forall {{i}} \in \mathcal{K}. \label{eq_e}
\end{gather}

\begin{lemma} \label{lem_psi_relation1}
Let agent $k \in \mathcal{K}$ and agent ${{j}} \in \mathcal{B}^{{k}}$. For any given prescription strategy $\boldsymbol{\psi}^{{k}}$ of agent ${{k}}$, there exists a positional relationship $e^{[{{j}},{{k}}]}$ such that a prescription strategy $\boldsymbol{\psi}^{{j}}$ of agent ${{j}}$ generated using \eqref{eq_e} yields:
\begin{align}
   \text{\emph{1. }} &{\Gamma}_t^{[{{k}},{{i}}]}(L_t^{[{{i}},{{i}}]}) = {\Gamma}_t^{[{{j}},{{i}}]}(L_t^{[{{i}},{{i}}]}),   \text{ \emph{if} } {{i}} \in \mathcal{B}^{{j}}, \nonumber \\
   \text{\emph{2. }} &{\Gamma}_t^{[{{k}},{{i}}]}(L_t^{[{{i}},{{i}}]}) = {\Gamma}_t^{[{{j}},{{i}}]}(L_t^{[{{i}},{{j}}]}),   \text{ \emph{if} } {{i}} \in \mathcal{B}^{{k}}, {{i}} \not\in \mathcal{B}^{{j}}, \nonumber \\
   \text{\emph{3. }} &{\Gamma}_t^{[{{k}},{{i}}]}(L_t^{[{{i}},{{k}}]}) = {\Gamma}_t^{[{{j}},{{i}}]}(L_t^{[{{i}},{{j}}]}),   \text{ \emph{if} } {{i}} \not\in \mathcal{B}^{{k}}.
    \label{lem_3_condition}
\end{align}
\end{lemma}

\begin{proof}
Let $g_t^{{i}}$ denote the control policy of agent $i \in \mathcal{K}$ at time $t$. To prove the result, we construct $g_t^{{i}}$ and the prescription strategy ${\psi}_t^{{j}}$ for three cases, given a prescription strategy $\boldsymbol{\psi}^{{k}}$.

\begin{enumerate}
\item If ${{i}} \in \mathcal{B}^{{j}}$, the control policy $g_t^{{i}}$ can be constructed from the prescription strategy ${\psi}_t^{[{{k}},{{i}}]}$, namely,
\begin{align}
  g_t^{{i}}(A_t^{{i}},{L}_t^{[{{i}},{{i}}]}) &= {\psi}_t^{[{{k}},{{i}}]}(A_t^{{i}})({L}_t^{[{{i}},{{i}}]}).
\end{align}
From \eqref{eq_gen_pres} we have
\begin{gather}
    \Gamma_t^{[{{j}},{{i}}]} = \psi_t^{[{{j}},{{i}}]}(A_t^{{i}}), \; \forall {{i}} \in \mathcal{B}^{{j}},
\end{gather}
and thus,
\begin{gather}
    \psi_t^{[{{j}},{{i}}]}(A_t^{{i}})(L_t^{[{{i}},{{i}}]}) = g_t^{{i}}(A_t^{{i}},{L}_t^{[{{i}},{{i}}]}).
\end{gather}
Hence,
\begin{gather}
    {\psi}_t^{[{{k}},{{i}}]}(A_t^{{i}})({L}_t^{[{{i}},{{i}}]}) = \psi_t^{[{{j}},{{i}}]}(A_t^{{i}})(L_t^{[{{i}},{{i}}]}). \label{eq_e_1}
\end{gather}

\item If ${{i}}\in \mathcal{B}^{{k}}$ and ${{i}} \not \in \mathcal{B}^{{j}}$, the control policy $g_t^{{i}}$ can be constructed by the prescription strategy ${\psi}_t^{[{{k}},{{i}}]},$ namely,
\begin{align}
   g_t^{{i}}(A_t^{{i}},{L}_t^{[{{i}},{{i}}]}) = {\psi}_t^{[{{k}},{{i}}]}(A_t^{{i}})({L}_t^{[{{i}},{{i}}]}).
\end{align}
From \eqref{eq_gen_pres} we have
\begin{gather}
    \Gamma_t^{[{{j}},{{i}}]} = \psi_t^{[{{j}},{{i}}]}(A_t^{{j}}), \; \forall {{i}} \not\in \mathcal{B}^{{j}}.
\end{gather}
Thus,
\begin{gather}
    {\psi}_t^{[{{j}},{{i}}]}(A_t^{{j}})({L}_t^{[{{i}},{{j}}]}) = g_t^{{i}}(A_t^{{j}},{L}_t^{[{{i}},{{j}}]}) = g_t^{{i}}(A_t^{{i}},{L}_t^{[{{i}},{{i}}]}).
\end{gather}
Hence,
\begin{align}
    {\psi}_t^{[{{k}},{{i}}]}(A_t^{{i}})({L}_t^{[{{i}},{{i}}]})
    ={\psi}_t^{[{{j}},{{i}}]}(A_t^{{j}})({L}_t^{[{{i}},{{j}}]}). \label{eq_e_2}
\end{align}

\item If ${{i}}\not\in\mathcal{B}^{{k}}$, the control policy $g_t^{{i}}$ can be constructed by the prescription strategy ${\psi}_t^{[{{k}},{{i}}]},$ namely,
\begin{align}
    g_t^{{i}}(A_t^{{k}},{L}_t^{[{{i}},{{k}}]}) =
   {\psi}_t^{[{{k}},{{i}}]}(A_t^{{k}})({L}_t^{[{{i}},{{k}}]}).
\end{align}
From \eqref{eq_gen_pres} we have
\begin{gather}
    \Gamma_t^{[{{j}},{{i}}]} = \psi_t^{[{{j}},{{i}}]}(A_t^{{j}}), \; \forall {{i}} \not\in \mathcal{B}^{{j}}.
\end{gather}
Thus,
\begin{align}
    {\psi}_t^{[{{j}},{{i}}]}(A_t^{{j}})({L}_t^{[{{i}},{{j}}]}) &= g_t^{{i}}(A_t^{{j}},{L}_t^{[{{i}},{{j}}]}) \nonumber \\
    &= g_t^{{i}}(A_t^{{k}},{L}_t^{[{{i}},{{k}}]}).
\end{align}
Hence,
\begin{align}
    {\psi}_t^{[{{k}},{{i}}]}(A_t^{{k}})({L}_t^{[{{i}},{{k}}]})=
    {\psi}_t^{[{{j}},{{i}}]}(A_t^{{j}})({L}_t^{[{{i}},{{j}}]}). \label{eq_e_3}
\end{align}
\end{enumerate}

To complete the proof, note that we can define a positional relationship $e^{[{{j}},{{k}}]} : \Psi^{{k}} \rightarrow \Psi^{{j}}$ with $e^{[{{j}},{{k}}]}=(e_1^{[{{j}},{{k}}]},\ldots,e_T^{[{{j}},{{k}}]})$ such that \eqref{eq_e} implies  \eqref{eq_e_1}, \eqref{eq_e_2} and \eqref{eq_e_3} with ${{j}} \in \mathcal{B}^{{k}}$.
\end{proof}

\begin{lemma} \label{lem_psi_relation2}
Let agents $k, j\in \mathcal{K}$ with ${{j}} \not\in \mathcal{B}^{{k}}$. For any given prescription strategy $\boldsymbol{\psi}^{{k}}$ of agent ${{k}}$, there exists a positional relationship $e^{[{{j}},{{k}}]}$ such that a prescription strategy $\boldsymbol{\psi}^{{j}}$ of agent ${{j}}$ generated from \eqref{eq_e} yields:
\begin{align}
    \text{\emph{1. }} &{\Gamma}_t^{[{{k}},{{i}}]}(L_t^{[{{i}},{{i}}]}) = {\Gamma}_t^{[{{j}},{{i}}]}(L_t^{[{{i}},{{i}}]}),   \text{ \emph{if} } {{i}} \in \mathcal{B}^{{k}}, \nonumber \\
    \text{\emph{2. }} &{\Gamma}_t^{[{{k}},{{i}}]}(L_t^{[{{i}},{{k}}]}) = {\Gamma}_t^{[{{j}},{{i}}]}(L_t^{[{{i}},{{i}}]}),   \text{ \emph{if} } {{i}} \in \mathcal{B}^{{j}}, {{i}} \not\in \mathcal{B}^{{k}}, \nonumber \\
    \text{\emph{3. }} &{\Gamma}_t^{[{{k}},{{i}}]}(L_t^{[{{i}},{{k}}]}) = {\Gamma}_t^{[{{j}},{{i}}]}(L_t^{[{{i}},{{j}}]}),   \text{ \emph{if} } {{i}} \not\in \mathcal{B}^{{j}}.
    \label{lem_4_condition}
\end{align}
\end{lemma}

\begin{proof}
The proof is very similar to the proof of Lemma \ref{lem_psi_relation1}. It is omitted due to space limitations.
\end{proof}

To this end, we consider a positional relationship function $e^{[{{j}},{{k}}]}$ from every ${{k}} \in \mathcal{K}$ to every position ${{j}} \in \mathcal{K}$ which satisfies the properties in Lemmas \ref{lem_psi_relation1} and \ref{lem_psi_relation2}. This implies that for any two agents ${{k}}$ and ${{j}}$, we have the relation,
\begin{gather}
    U_t^{{j}} = {\Gamma}_t^{[{{j}},{{j}}]}(L_t^{[{{j}},{{j}}]}) =
    \begin{cases}
        {\Gamma}_t^{[{{k}},{{j}}]}(L_t^{[{{j}},{{k}}]}),   \text{ if } {{j}} \not\in \mathcal{B}^{{k}}, \\
        {\Gamma}_t^{[{{k}},{{j}}]}(L_t^{[{{j}},{{j}}]}),   \text{ if } {{j}} \in \mathcal{B}^{{k}}.
    \end{cases}
    \label{eq_u_with_another_psi}
\end{gather}

\section{RESULTS}

\subsection{Equivalent Prescription Problems}

Lemmas \ref{lem_psi_g_relation} through \ref{lem_psi_relation2} lead to \eqref{eq_u_with_another_psi}. This implies that the control action $U_t^j$ for agent $j \in \mathcal{K}$ can be equivalently obtained through the prescription function $ {\Gamma}_t^{[{{k}},{{j}}]}$ of any other agent ${{k}} \in \mathcal{K}$, if the corresponding inaccessible information is available.
Using \eqref{eq_u_with_another_psi}, we can write the cost incurred by the system at time $t$ as
\begin{align}
    c_t(X_t,U_t^1,\ldots,&U_t^K) \nonumber \\
    =: c_t\big(&X_t,{\Gamma}_t^{[{{k}},1]}({L}_t^{[1,{{k}}]}),\ldots,{\Gamma}_t^{[{{k}},{{k}}]}({L}_t^{[{{k}},{{k}}]}), \nonumber \\
    &{\Gamma}_t^{[{{k}},{{k+1}}]}({L}_t^{[{{k+1}},{{k+1}}]}),\ldots,{\Gamma}_t^{[{{k}},K]}({L}_t^{[{{K}},K]})\big). \label{o_cost}
\end{align}

We can then reformulate Problem 1 in terms of the prescription strategy of any agent ${{k}}$. The optimization problem is to select the optimal prescription strategy $\boldsymbol{\psi}^{{*k}} \in \Psi^{{k}}$ that minimizes the performance criterion given by the total expected cost:
\begin{multline}
    \textbf{Problem 2:}~~~\mathcal{J}^{{k}}(\boldsymbol{\psi}^{{k}}) =  \\
    \mathbb{E}^{\boldsymbol{\psi}^{{k}}} \Big[\sum_{t=0}^T{c_t\big(X_t,{\Gamma}_t^{[{{k}},1]}({L}_t^{[1,{{k}}]}),\ldots,{\Gamma}_t^{[{{k}},{{k}}]}({L}_t^{[{{k}},{{k}}]})},  \\
    {\Gamma}_t^{[{{k}},{{k+1}}]}({L}_t^{[{{k+1}},{{k+1}}]}),\ldots,{\Gamma}_t^{[{{k}},K]}({L}_t^{[{{K}},K]})\big)\Big]. \label{per_cri_2}
\end{multline}

The task of deriving optimal prescription strategy $\boldsymbol{\psi}^{{*k}},$ and subsequently, the complete prescription  $\Theta_t^{{k}}$ for agent ${{k}}$ is assigned to a fictitious designer that can only access memory $M_t^{{k}}$. Note that this maintains the decentralized nature of the problem as the strategies are implemented by the agents in real time with asymmetric and incomplete information. Now, we show the equivalence between the two problems.

\begin{lemma} \label{lem_equivalence}
For any agent ${{k}} \in \mathcal{K}$, Problem 2 is equivalent to Problem 1.
\end{lemma}

\begin{proof}
Eq. \eqref{o_cost} implies that the performance criterion $\mathcal{J}^{{k}}(\boldsymbol{\psi}^{{k}})$ in \eqref{per_cri_2} is equal to the performance criterion $\mathcal{J}(\boldsymbol{g})$ in \eqref{per_cri}. Thus, given the optimal prescription strategy $\boldsymbol{\psi}^{{*k}}$, from Lemmas \ref{lem_psi_g_relation} and \ref{lem_psi_g_relation_inv}, we can derive the optimal policy $\boldsymbol{g}^*$ in Problem 1. 
\end{proof}

Next, we present a state sufficient for input-output mapping in Problem 2 for agent ${{k}}$ following the exposition presented in \cite{mahajan2008sequential}.

\begin{lemma} \label{state_suff_k}
A state sufficient for input-output mapping for agent ${{k}} \in \mathcal{K}$ is
\begin{equation}
    S^{{k}}_t := \left\{X_{t}, L_t^{[1,{{k}}]},\ldots,L_t^{[{{k-1}},{{k}}]},L_t^{[{{k}},{{k}}]},\ldots,L_t^{[{{K}},K]}\right\}. \label{S_k}
\end{equation}
\end{lemma}

\begin{proof}
The state $S^{{k}}_t$ satisfies the three properties stated by Witsenhausen \cite{witsenhausen1976some}:

1) There exist functions $\{\hat{f}^{{k}}_t$: $t \in \mathcal{T}\}$ such that
  \begin{equation}
	S^{{k}}_{t+1} = \hat{f}^{{k}}_{t}(S^{{k}}_t, W_t, V_{t+1}^{1:K}, \Theta_t^{{k}}). \label{eq_St_k1}
  \end{equation}
  
2) There exist functions $\{\hat{h}^{{k}}_t$: $t \in \mathcal{T}\}$ such that
    \begin{equation}
	Z^{{k}}_{t+1} = \hat{h}^{{k}}_{t}(S^{{k}}_{t},\Theta_t^{{k}},V_{t+1}^{1:K}). \label{eq_St_k2}
  \end{equation}
  
3) There exist functions $\{\hat{c}^{{k}}_t$: $t \in \mathcal{T}\}$ such that
      \begin{align}
	c_t(X_t,U_t^{1:K}) &= \hat{c}^{{k}}_t(S^{{k}}_t,\Theta^{{k}}_t). \label{eq_St_k3}
  \end{align}
The three equations above can each be verified by substitution of variables on the LHS. The complete proof can be found in \cite{2019Aditya_arXiv}.
\end{proof}

\subsection{The Information States}

From the designer's point of view, the system behaves as a Partially Observed Markov Decision Process (POMDP) with state $S_t^{{k}}$, control input $\Theta_t^{{k}}$, output $Z_t^{{k}}$ (with $Z^{{k}}_{0:t} = A^{{k}}_t$) and cost $\hat{c}^{{k}}_t(S^{{k}}_t,\Theta^{{k}}_t)$ at time $t$.
The difference is that the prescription functions $\Gamma_t^{[{{k}},{{j}}]}$, ${{j}} \in \mathcal{B}^{{k}},$ are generated as functions of the accessible information $A_t^{{j}}$ instead of $A_t^{{k}}$. Thus, structural results for POMDPs cannot be directly applied to Problem 2. Before proceeding to structural results, we define the \textit{information state} for agent ${{k}}$.

\begin{definition}
Let $S_t^{{k}}$ be the state, $A_{t}^{{k}}$ the accessible information, and $\Theta^{{k}}_{0:t-1}$  the control inputs at time $t$ an agent ${{k}} \in \mathcal{K}$. The \textit{information state} is defined as a probability distribution $\Pi^{{k}}_t$ that takes values in the possible realizations $\mathscr{P}^{{k}}_t := \Delta(\mathcal{S}^{{k}}_t)$ such that,
\begin{equation}
\Pi^{{k}}_t(s^{{k}}_t) := \mathbb{P}^{\boldsymbol{\psi}^{{k}}}(S^{{k}}_t = s^{{k}}_t \big|A^{{k}}_t, \Theta^{{k}}_{0:t-1}).
\end{equation}
\end{definition}

Due to space limitation, the proofs of the following three properties are omitted but can be found in \cite{2019Aditya_arXiv}.
The first property establishes that the information state $\Pi_t^{{k}}$ is independent from the prescription strategy $\boldsymbol{\psi}^{{k}}$. 

\begin{lemma} \label{pi_k_1} 
At time $t$, there exists a function $F_t^{{k}}$ independent from the prescription strategy $\boldsymbol{\psi}^{{k}}$ such that
\begin{equation}
    \Pi_{t+1}^{{k}} = F_{t+1}^{{k}}(\Pi_t^{{k}},\Theta_t^{{k}},Z_{t+1}^{{k}}).
\end{equation}
\end{lemma}

The second property of the information state $\Pi_t^{{k}}$ is that its evolution is Markovian.

\begin{lemma} \label{pi_k_2} 
The evolution of the information state $\Pi_t$ is a controlled Markov Chain with $\Theta^{{k}}_{t}$ as the control action at time $t$
\begin{align}
\mathbb{P}^{\boldsymbol{\psi}^{{k}}}(\Pi_{t+1}^{{k}}|A_t^{{k}},\Pi^{{k}}_{0:t}, \Theta^{{k}}_{0:t}) = \mathbb{P}^{\boldsymbol{\psi}^{{k}}}(\Pi_{t+1}^{{k}}|\Pi^{{k}}_{t}, \Theta^{{k}}_{t}).
\end{align}
\end{lemma}

The third property of the information state $\Pi_t^{{k}}$ is that the expected cost incurred by the system at time $t$ can be written as a function of $\Pi_t^{{k}}$.

\begin{lemma} \label{pi_k_3} 
There exists a function $C^{{k}}_t$, independent of the prescription strategy $\boldsymbol{\psi}^{{k}}$, such that
\begin{equation}
    \mathbb{E}^{\boldsymbol{\psi}^{{k}}}\big[\hat{c}^{{k}}_t(S^{{k}}_t,\Theta^{{k}}_t)|A^{{k}}_t,\Theta^{{k}}_{0:t}\big] = C_t^{{k}}(\Pi_t^{{k}},\Theta_t^{{k}}).
\end{equation}
\end{lemma}

In Lemmas \ref{pi_k_1} through \ref{pi_k_3}, we established that the information state $\Pi_t^{{k}}$ evolves as a controlled Markov chain with control inputs $\Theta_t^{{k}}$.

\subsection{Structural Results}

We start by presenting a structural result for agent $K$. By definition, the set of agents beyond agent $K$ contains only agent $K$, i.e., $\mathcal{B}^{K} = \{K\}$. Using \eqref{eq_gen_pres}, this implies that for all agents ${{k}} \in \mathcal{K}$, the prescription component $\Gamma^{[{{K}},{{k}}]}_t$ is a function of the accessible information $A_t^{{[K]}}$. This leads to the following result derived in \cite{14} through the common information approach.

\begin{lemma}\label{lem_case_K}
Consider agent $K$. There exists an optimal prescription strategy $\boldsymbol{\psi}^{*K}$ of the form
\begin{equation}
    \Gamma_t^{*[{{K}},{{k}}]} = \psi_t^{*[{{K}},{{k}}]}(\Pi_t^{K}), \label{eq_common_info}
\end{equation}
that optimizes the performance criterion \eqref{per_cri_2} in Problem 2.
\end{lemma}

We know that for any two agents ${{k}} \in \mathcal{K}$ and ${{j}} \in \mathcal{B}^{{k}}$, we have $A_t^{{j}} \subset A_t^{{k}}$. Given the accessible information $A_t^{{k}}$ and the optimal prescription strategy $\boldsymbol{\psi}^{*{{k}}}$, agent ${{k}}$ can derive the optimal complete prescriptions $\Theta_t^{*{{j}}}$ for every ${{j}} \in \mathcal{B}^{{k}}$. This leads to the following structural result, proved in  \cite{2019Aditya_arXiv}.

\begin{theorem} \label{struct_result}
Consider agent ${{k}} \in \mathcal{K}$. There exists an optimal prescription strategy $\boldsymbol{\psi}^{*{{k}}}$ of the form
\begin{align}
\Gamma_t^{*[{{k}},{{j}}]}(\cdot) =
\begin{cases}
{\psi}_t^{*[{{k}},{{j}}]}(\Pi^{{k}}_t,\ldots,\Pi_t^{K}), \text{\emph{ if }} {{j}} \not\in\mathcal{B}^{{k}}, \\
{\psi}_t^{*[{{k}},{{j}}]}(\Pi^{{j}}_t,\ldots,\Pi_t^{K}), \text{\emph{ if }} {{j}} \in \mathcal{B}^{{k}},
\end{cases} \label{eq_struct_result}
\end{align}
that optimizes the performance criterion \eqref{per_cri_2} in Problem 2.
\end{theorem}

\subsection{A Comparison with Existing Approaches}

Among the existing approaches, the person-by-person approach and the designer's approach do not yield the kind of structural results presented in this paper. The graphical approach presented in \cite{mahajan2015algorithmic} has similarities with the prescription approach, but, it applies only to problems where agents have perfect observations. The common information approach in \cite{14} can be applied to the problem with the word-of-mouth communication structure to obtain the structural result presented in Lemma \ref{lem_case_K}, since by definition, the accessible information $A_t^{{K}}$ is the common information in the system. Thus, the control action for agent $k \in \mathcal{K}$ is given by
\begin{gather}
     U_t^{{*k}} = \Gamma_t^{*[{{K}},{{k}}]}(L_t^{[{{k}},{{K}}]}).
\end{gather}

In contrast, we see that when we consider Problem 2 for agent ${{k}}$, the control action of agent ${{k}}$ is given by
\begin{gather}
    U_t^{{*k}} = \Gamma_t^{*[{{k}},{{k}}]}(L_t^{[{{k}},{{k}}]}).
\end{gather}
Now, from \eqref{eq_partition}, we note that,
\begin{gather}
    A_t^{{k}} \cup L_t^{[{{k}},{{k}}]} = A_t^{{K}} \cup L_t^{[{{k}},{{K}}]}, \label{eq_comparison_1}
\end{gather}
and from \eqref{ainfo_prop_1} we have the relation,
\begin{gather}
    A_t^{{K}} \subset A_t^{{k}}, \label{eq_comparison_2}
\end{gather}
because $K \in \mathcal{B}^{{k}}$ for all ${{k}} \in \mathcal{K}$. Then, \eqref{eq_comparison_1} and \eqref{eq_comparison_2} imply,
\begin{gather}
    L_t^{[{{k}},{{k}}]} \subset L_t^{[{{k}},{{K}}]}.
\end{gather}

Thus, the prescription functions generated through the prescription approach have an equal or smaller domain when compared with those generated through the common information approach.

\section{CONCLUSIONS}

In this paper, we introduce a network of agents with a word-of-mouth communication structure, and analyze it using the prescription approach, which yielded some desired properties. We showed that the structural result derived through the common information approach can be considered as the outcome of one reformulations using the prescription approach. Finally, we provided, without proof, a preliminary structural result arising from the prescription approach. A direction for future research should seek to extend these results for a broader class of decentralized systems.

\bibliographystyle{ieeetr}
\bibliography{References}

\end{document}